\documentclass[a4paper,12pt]{amsart}
\usepackage[v2,all]{xy}
\usepackage[margin=1in]{geometry}
\usepackage{youngtab,enumitem}
\usepackage[colorlinks=true]{hyperref}
\newcommand{\comment}[1]{%\medskip\par\noindent\fbox{\begin{minipage}{\textwidth}#1\end{minipage}}\medskip\par
}
\newcommand{\shortcomment}[1]{%\fbox{#1}
}

\newtheorem{thm}{Theorem}[section]
\newtheorem{prop}[thm]{Proposition}
\newtheorem{lem}[thm]{Lemma}
\newtheorem{cor}[thm]{Corollary}
\newtheorem{defn}[thm]{Definition}

\theoremstyle{definition}
\newtheorem{eg}[thm]{Example}
\newtheorem{rem}[thm]{Remark}

\newcommand{\C}{\mathbb{C}}
\newcommand{\La}{\langle}
\newcommand{\Ra}{\rangle}

\newcommand{\mf}[1]{\mathfrak{#1}}

\begin{document}

\title{The Berkovits Complex and Semi-free Extensions of Koszul Algebras}
\author[Galvez]
{Imma G\'alvez}
\address{Departament de Matem\`atica Aplicada III\\ Universitat Polit\`ecnica de Catalunya\\ Escola
d'Enginyeria de Terrassa\\ Carrer Colom 1\\ 08222 Terrassa (Barcelona), Spain
}
\email{m.immaculada.galvez@upc.edu}

\author[Gorbounov]{Vassily Gorbounov}
\address{Institute of Mathematics\\
University of Aberdeen\\
Fraser Noble Building, King's College, Aberdeen AB24 3UE, United Kingdom}
\email{vgorb@maths.abdn.ac.uk}

\author[Shaikh]{Zain Shaikh}

\address{Fakult\"at f\"ur Elektrotechnik, Informatik und Mathematik\\
Institut f\"ur Mathematik\\
Warburger Str. 100, 33098 Paderborn}
\email{zain@math.upb.de}

\author[Tonks]
{Andrew Tonks}
\address{Department of Mathematics\\University of Leicester\\University Road\\Leicester LE1 7RH\\United Kingdom}
\email{apt12@leicester.ac.uk }

\maketitle
\begin{abstract}
  In his extension \cite{Ber} of W. Siegel's ideas on string
  quantization, N. Berkovits made several observations
  %of a purely algebraic nature
  which deserve further study and development.
  Indeed, interesting accounts of this work have already appeared in the
  mathematical literature~\cite{GKhR,MSh} and in a different guise due to Avramov. In this paper we bridge between these three approaches, by providing a complex that is useful in the calculation of some homologies.
\end{abstract}
\tableofcontents
\section{Introduction}

This paper began with an observation of the importance of Koszul
duality theory for commutative algebras in the work of
Berkovits~\cite{Ber} on string quantization and string/gauge theory
duality.  In Berkovits' paper the commutative algebra in question is
the projective coordinate algebra $S$ of the orthogonal Grassmannian $OG(5,10)$,
related to the spinor representation of the group $SO(10,{\mathbb
  C})$.  This is a commutative quadratic algebra, and for such algebras it is an easy consequence of
the definition that its Koszul dual is the universal enveloping
algebra of a graded Lie superalgebra $L=\bigoplus_{i\geq1}
L_i$.% (see (\ref{defn:L}) below).

In a very interesting paper Movshev--Schwarz \cite{MSh} observed that the algebra of syzygies of $S$ is isomorphic to the
cohomology $H^*(L_{\geq2},\C)$ of the graded Lie superalgebra
$L_{\geq2}=\bigoplus_{i\geq 2} L_i$.  
They proposed a further complex in an attempt to describe an off-shell formulation of $N=1$, $D=10$ Yang--Mills theory, 
and posed the question of calculating its homology.
Their construction is an iteration of the Koszul homology of a
sequence of elements of an algebra, first applied to the algebra $S$
and then applied to its syzygies.  
% Koszul homology with respect to a sequence
%of elements 
%of a commutative algebra $A$.  
Namely,
if $A$ is
finitely generated and presented as ${\mathbb C}[a_1,\dots,a_n]/I$, the
Koszul homology with respect to the sequence $\{a_i\}$ is called the
algebra of syzygies of $A$. The generators $\Gamma_j$ of the ideal $I$
represent syzygies in the lowest degree.  The Koszul homology of the
algebra of syzygies with respect to the sequence $\{\Gamma_j\}$ is the
algebra studied by Movshev--Schwarz.
%Further examples are supplied by the orbits of the highest
%weight vector in an irreducible representation of a semisimple complex
%Lie group \cite{GKhR}. 
 
The same construction has appeared in the study of {\em deviations} of a local ring (see \cite{Avr} and section 2.3 below). These are a sequence of integers that are attached to a local ring, that measure how far it is from being regular or a complete intersection. In calculating them, Avramov constructs complexes that are analogous to Movshev--Schwarz.  
A consequence of Avramov's work is that if $A$ is Koszul, the Berkovits homology
can be described explicitly in terms of the graded Lie superalgebra $L$.
Namely, 
for an arbitrary finitely generated commutative Koszul algebra, the Berkovits homology 
is isomorphic to the cohomology $H^*(L_{\geq 3},\C)$.
This result was also shown by Movshev--Schwarz for the case $A=S$. Further, they suggested an algebraic technique to study the Berkovits homology of the algebra $S$ by constructing an equivalent but ``smaller'' complex. This complex does not seem to appear as an accident; we believe that it is naturally attached to a commutative Koszul algebra. We construct such a ``smaller'' complex for a Koszul algebra attached to the Pl\"ucker embedding of the Grassmannian $G(2,5)$. %In section 5 we prove that their method indeed produces a shorter complex when algebra of syzygies is also Koszul. 

At this point we would like to mention a connection with a conjecture made by Avramov:
Conjecture C$_{10}$ of \cite{Av} states that if $A$ is not a complete intersection, 
then the appropriate Lie superalgebra $L$ contains a free nonabelian graded Lie subalgebra.
%It makes sense to extend this conjecture to the category of graded augmented algebras. 
The algebra $S$ above provides an example confirming the conjecture, according to \cite{HS,MSh}.
%and our work in section 5 extends this to a general method which may be of interest in this area. 
%Our work in section 5 
%We show that the bar duality method can be used to build a bunch of examples confirming the conjecture by studying the coordinate algebra of the Grassmanian $G(2,5)$. 
%In a very elementary way we show by the bar duality method that this algebra confirms the conjecture. 
%One can set a task of proving that the coordinate algebras of the homogeneous projective spaces we consider above always agree with Avramov's conjecture. 
%We plan to return to it in a further publication.
%This is important from the point of view
%of gauge theory because it implies that the solution to the
%Euler-Lagrange equations for an abelian gauge group can be deformed to
%the solutions for a non-abelian gauge group.
Indeed, the ``smaller'' complex allows one to calculate that $H^2(L_{\geq3},\C)$ is trivial. 
Our construction of a ``smaller'' complex allows us to show that the Lie superalgebra $L_{\geq3}$ is free for the projective coordinate algebra of $G(2,5)$, providing another confirmation of C$_{10}$.
%Movshev and Schwarz, in a very interesting paper \cite{MSh}, proved that $L_{\geq 3}$ is a free Lie superalgebra on an
%infinite set of generators for the algebra $A$ considered above
%in the case of the coordinate algebra of
%pure spinors, see also \cite{}.  

%It seems to work for a large class of algebras, although the details of the construction
%are becoming heavier. The result of this paper
%can be considered as another approach to the ideas of \cite{MSh}.
%This extension is furnished in our paper as a deformation of the Chevalley differential of the Lie algebra Koszul dual to the original algebra $A$.
%Such an extension have a simple formula and provides a complex equivalent to the one constructed in \cite{MSh}. The precise relation of our approach and the Movshev-Schwarz 
%We plan to explore the precise relation between the two approaches  to calculating the algebra of generalized syzygies in a separate publication.

%Another aspect of the work of \cite{Ber} and \cite{MSh} w

The algebra of syzygies of $S$ is in fact non-quadratic, and yet the ``smaller'' complex still exists. So, we believe that our construction can be extended to a more general case, but requires (to our knowledge) a substantial extension of the notion of Koszulness to non-quadratic algebras.

The paper is organised as follows: in Section \ref{sec:prelim} we discuss some preliminaries including Koszul algebras and Lie superalgebras; in Section \ref{sec:ber} we define the Berkovits complex and how it arises as a semi-free extension of the initial algebra. The main theorem of the paper is Theorem \ref{thm:syz}, which shows that for any commutative Koszul algebra, its Berkovits homology is isomorphic to $H^*(L_{\geq3},\C)$. It is in Section \ref{sec:bv} where we construct the ``small'' complex and prove that it calculates the Berkovits homology when we have a Koszul algebra attached to the Pl\"ucker embedding of the Grassmannian $G(2,5)$ (Corollary \ref{cor:bv}).

The first author was partially supported by Spanish Government grants
MTM2010-15831, 
MTM2010-20692, 
MTM2012-38122-C03-01 and 
MTM2013-42178-P 
and Catalan Government grants 
SGR1092-2009 and
SGR634-2014,
and the fourth author by 
MTM2010-15831 and MTM2013-42178-P.
All the authors visited the Max Planck Institute while working on this
project and are grateful to the Institute for excellent working
conditions. 
The authors also thank to L.A.\ Bokut, A.\ Losev and M. Movshev for very useful discussions and advice, 
and A.\ Conca and S.\ Iyengar for pointing us to Avramov's commutative algebra constructions in \cite{Avr}. We are also grateful to the referee for pointing out an error in \cite[Theorem 4.4.1]{GKhR}, which we initially used to construct the smaller complex for a larger class of algebras.
Of course we also are very grateful to Vadim Schechtman with whom this project originally was
started and to whom we dedicate it with our best wishes.

\section{Preliminaries on commutative Koszul algebras}\label{sec:prelim}

We fix our ground field $k=\C$. 
Unless otherwise stated, all algebras are graded $A=\bigoplus_{i\ge0}A_i$ and locally finite dimensional, i.e. ${\rm dim}\;A_i<\infty$ for all $i\geq0$. Further, we assume that $A_0=\C$, so that $A=\C\oplus A_+$ is augmented, with augmentation ideal $A_+=\bigoplus_{i>0}A_i$.

\subsection{Koszul Algebras}\label{sec:kosalg}

\begin{defn}
Let $A$ be a quadratic algebra defined by a presentation
$A=T(V)/Q$, where $V$ is a finite dimensional vector space of generators in degree one, and $Q$ is a two-sided ideal generated by quadratic elements.
  The \emph{Koszul dual algebra} $A^!$ is the graded algebra defined as \[A^!=T(V^*)/Q^\perp,\] where $V^*$ is concentrated in degree one and $Q^\perp\subset V^*
  \otimes V^*$ is the annihilator of $Q$. 
\end{defn}
\noindent Clearly $A^{!!}=A$.
\medskip

The \emph{Koszul complex} $(K(A),d_A)$ of a quadratic algebra $A$ is defined by the sequence of right $A\otimes A^!$-modules 
\[K_p(A):=A\otimes (A^!)^*_p\qquad\qquad\qquad\quad (p\geq0).\]
Here $A$ is a right $A$-module via multiplication and $(A^!)^*$ is a right $A^!$-module via 
$$(\varphi b)(c):=\varphi(bc), \qquad (\varphi\in (A^!)^*,\;b,c\in A^!).$$
The differential $d_A$ is the action of the identity $\mathrm{id}_V\in \mathrm{Hom}(V,V)= V \otimes V^* \subset A \otimes A^!$.  It has degree $-1$, and satisfies $d_A^2=0$ since $A$ is quadratic.

Recall that since $A$ is graded, the algebra ${\rm Ext}_A(\C,\C)$ is bigraded, and decomposes as
$$
{\rm Ext}_A^i(\C,\C)=\bigoplus_{j\geq 0} {\rm Ext}_A^{ij}(\C,\C),
$$
where $i$ is the cohomological grading and $j$ is the grading descending from $A$.
\begin{defn}\label{defn:koszul}
An augmented algebra $A$ is a \emph{Koszul algebra} if ${\rm Ext}_A^{ij}(\C,\C)=0$ whenever $i\neq j$.
\end{defn}

The following theorem gives some equivalent definitions of a Koszul algebra. A proof may be found, for example, in \cite[Chapter 2, Theorem 4.1]{PP}.
\begin{thm}
Let $A$ be a quadratic algebra. Then the following are equivalent:
\begin{enumerate}
\item $A$ is Koszul,
\item $A^!$ is Koszul,
\item $\mathrm{Ext}_A(\C,\C)\cong A^!$ as algebras with ${\rm Ext}_A^i(\C,\C)=A^!_i$,
\item $(K(A),d_A)$ is a resolution of $\C$.
\end{enumerate}
\end{thm}
%%%% TODO !!
%  K(A) ~ C <=> K(A!) ~ C ? obvious ? KA=K(A!) ? grading ? 

\begin{eg}\label{eg:SV}
Consider the symmetric algebra, $S(V)$. It is a quadratic algebra and the ideal of relations $Q$ inside $T(V)$ is precisely $(u\otimes v-v\otimes u \;|\; u,v\in V)$. The annihilator of $Q$ inside $T(V^*)$ is generated by tensors of the form $u^*\otimes v^*+v^*\otimes u^*$ and so $S(V)^!=\bigwedge V^*$.
Therefore the Koszul complex of $S(V)$ has graded components 
\begin{equation}
K_p(S(V))=S(V)\otimes {\bigwedge}^{\!p\;} W_1\qquad\qquad\qquad\quad (p\geq0)
\end{equation}
where $W_1$ is a copy of $V$ inside $(S(V)^!)^*$. 
%Each $K_p(S(V))$ is an $S(V)$-module. 
If $a_1,\dots,a_n$ and $\xi_1,\dots,\xi_n$ are bases of $V$ and $W_1$ respectively, then the differential can be written as
$$
\sum_ia_i\frac{\partial}{\partial \xi_i}.
$$
This makes $K(S(V))$ into a DG algebra over $S(V)$, with the obvious product structure.
It is a resolution of $\C$ (see \cite[Proposition VII.2.1]{MacLH} for a proof), so $S(V)$ and  $S(V)^!=\bigwedge V^*$ are Koszul algebras.
\end{eg}

\subsection{Lie Superalgebras}

\begin{defn}
  A \emph{Lie superalgebra} over $\C$ is a $\mathbb{Z}_2$-graded vector
  space $L=L_{\bar{0}}\oplus L_{\bar{1}}$ with a map
  $[\cdot,\cdot]:L\otimes L\to L$ of $\mathbb{Z}_2$-graded spaces, satisfying:
  \begin{enumerate}
  \item (anti-symmetry) $[x,y]=-(-1)^{|x||y|}[y,x]$ for all
    homogeneous $x,y\in L$,
  \item (Jacobi identity)
    $(-1)^{|x||z|}[x,[y,z]]+(-1)^{|y||x|}[y,[z,x]]+(-1)^{|z||y|}[z,[x,y]]=0$
    for all homogeneous $x,y,z\in L$.
  \end{enumerate}
  Here  $|x|=i$ when $x\in L_{\bar{i}}$ for $i=0,1$, and
  an element $x$ in $L_{\bar{0}}$ or $L_{\bar{1}}$ is termed even or odd respectively.
We recover the familiar definition of a Lie algebra (over $\C$) in the case $L=L_{\bar{0}}$.

  A \emph{graded Lie superalgebra} is a Lie superalgebra $L$ together with
  a grading compatible with the bracket and supergrading. That is,
  $L=\bigoplus_{m\geq 1} L_m$ such that $[L_i,L_j]\subset L_{i+j}$
  and $L_{\bar{i}}=\bigoplus_{m\geq 1} L_{2m-i}$ for $i=0,1$.
\end{defn}
\begin{rem}
  In the literature a graded Lie superalgebra is sometimes called simply
  a graded Lie algebra. However, this could also refer to an ordinary Lie algebra $L=L_{\bar{0}}$ with a grading compatible with the bracket.
In
  order to avoid this ambiguity we prefer the terminology as in the
  definition.
\end{rem}

Let $V$ be a vector space with basis $a_1,\dots,a_n$ and consider a quadratic commutative algebra $A=T(V)/Q$. Then, $Q=C\oplus I$ where $C$ is the ideal $( u\otimes v - v\otimes u\;|\; u,v\in V)$ and $I=(\Gamma_1,\dots,\Gamma_m)\subset S^2(V)$ so that we can write $A=S(V)/I$. Further, $Q \supseteq C$ implies that $Q^\perp \subseteq C^\perp=S^2(V^*)$ and so $Q^\perp$ is generated by certain linear
combinations of anti-commutators $[a^*_i,a^*_j]=a^*_i a^*_j+
a^*_ja^*_i$. We can therefore describe the Koszul dual as the
universal envelope of a graded Lie
superalgebra,
\begin{equation}\label{defn:L}
  A^!=U(L),\qquad\qquad L=\bigoplus_{m\geq 1}L_m=\mathbb{L}(V^*)/J,
\end{equation}
where $\mathbb{L}$ is the free Lie superalgebra functor, the space of (odd) generators $V^*$ is concentrated in degree 1, and $J$ is the Lie ideal
with the same generators as $Q^\perp$ but viewed as linear
combinations of supercommutators.

\begin{defn}
  For $k\ge 2$, we define the graded Lie superalgebras
$$L_{\geq k}=\bigoplus\limits_{m\geq k}L_m.$$
\end{defn}

The interpretation of the algebras $L_{\geq k}$ for $k>2$ was outlined
in \cite{Ber}. In this paper, we concentrate on the case $k=3$.

%A minimal

\subsection{Deviations}

Consider the Hilbert series of the commutative algebra $A=S(V)/I$
\begin{equation}\label{poincare}
  H_A(t):=\sum_{r=0}^\infty \dim(A_r)\,t^r=\frac{h(t)}{(1-t)^{n}},
\end{equation}
where 
%${(r)}$ is the weight grading on $A$ induced from $S(V)$, so
%that $\dim(A^{(0)})=1$ and $\dim(A^{(1)})=\dim(V)=n$, and 
$h(t)$ is a
polynomial such that $h(0)=1$. Gauss' cyclotomic identity
\cite{Gauss} allows us to expand this rational function into an
infinite product
\begin{equation}
  \frac{h(t)}{(1-t)^{n}}=\prod_{s=1}^{\infty}(1-t^s)^{(-1)^s\,\varepsilon_s(A)},\label{gauss}
\end{equation}
where the exponents $\varepsilon_s(A)$ are integers known as the \emph{deviations} of $A$.

When $A$ is a Koszul algebra, the exponents $\varepsilon_s(A)$ give the dimensions of the graded components $L_s$
of the Lie superalgebra $L$ (see for example~\cite{Kang}), and in particular we have $\varepsilon_1(A)=n$.
Indeed,
by the Poincar\'e--Birkhoff--Witt theorem for the universal enveloping algebra of a graded Lie superalgebra,
the Hilbert series  of $A^!=U(L)$
is equal to
$$
H_{A^!}(t)
=\prod_{s=1}^{\infty}(1-(-t)^s)^{(-1)^{s-1}\dim(L_s)}
$$
and from the relation  $H_{A}(t)\,H_{A^!}(-t)=1$ it follows that
$\varepsilon_s(A)=\dim(L_s)$ in \eqref{gauss}.
See \cite[Section 2.2]{PP} for further details, and also Remark~\ref{eulerpoincare} below.

Dividing by \[\prod_{s=1}^{k-1}(1-t^s)^{(-1)^s\varepsilon_s(A)}\] we obtain the
Hilbert series of $L_{\geq k}$ as the following infinite product
$$\prod_{s=k}^{\infty}(1-t^s)^{(-1)^s\varepsilon_s(A)}.$$

\section{Koszul homology and Berkovits complex}

\subsection{Koszul homology and the algebra of syzygies}

Let $A=S(V)/I$ be a commutative quadratic algebra with $\{a_1,\dots,a_n\}$ a basis of $V$. 
 
\begin{defn}\label{defn:koshom}
  %Let $\{x_1,\dots,x_k\}$ be a sequence of elements of $A$. 
  The \emph{Koszul homology} of $A$ with respect to a sequence of elements
  $x_1,\dots,x_k$ in $A$ is the homology of the
  complex 
  \[
  A[W_1] := \left(A\otimes\bigwedge W_1,\:%d_K=
  \sum_{i=1}^k x_i\frac{\partial}{\partial\theta_i}\right)
  %(\xi_1,\dots,\xi_k)
  \] 
  where $A$ has homological degree zero and 
  $W_1$ is the vector space spanned by elements $\theta_1,\dots,\theta_k$ in homological degree one.
  %, and the Koszul differential is given by the formula
%$$\sum_{i=1}^k x_i\frac{\partial}{\partial\xi_i}$$

The \emph{algebra of syzygies} of $A$ is the Koszul homology of $A$ with respect to the sequence $\{a_1,\dots,a_n\}$.
 \end{defn}
 
The algebra of syzygies of $A$ is graded finite dimensional by
Hilbert's syzygy theorem (c.f. \cite[Theorem 1.1]{Eis}).
%, and the dimensions of the graded components are the coefficients of $h(t)$ in \eqref{poincare}.

\begin{rem}\label{rem:syzygy} We make the following trivial remarks:

\begin{enumerate}
\item Suppose $A$ is the symmetric algebra $S(V)$. Then $A^!=\bigwedge V^*$ and the complex $A[W_1]$ calculating the algebra of syzygies
of $A$ is precisely the Koszul complex, compare Example \ref{eg:SV}.
\item The complex $A[W_1]$ is a DG algebra. The algebra structure is given by
$$
(a \otimes \omega) \cdot (b \otimes \eta) = ab\otimes \omega \wedge \eta,
$$
where $a,b \in A$ and $\omega,\eta\in \bigwedge W_1$. By a theorem of Kadeishvili \cite{Kad}, the Koszul homology, and in particular the algebra of syzygies, inherits an $A_\infty$-algebra structure, which is unique up to (non-unique) isomorphism. We call this the $A_\infty$-\emph{minimal model} for $A[W_1]$ (the term minimal model will be reserved for something else).
\item When $A$ is graded, the Koszul homology is bigraded. In the case of Koszul homology with respect to a sequence of generators, the homological grading will be called the \textit{order} of the syzygies and the sum of the homological grading and the grading on $A$ will be called the \textit{degree} of the syzygies.
\end{enumerate}
\end{rem}

We observe that ${\rm Tor}^{S(V)}(A,\C)$ is also calculated by the complex $A[W_1]$, and so coincides with the algebra of syzygies. Indeed, $K(S(V))$ is a resolution of $\C$ by $S(V)$-modules, so ${\rm Tor}^{S(V)}(A,\C)$ is the homology of the complex
\begin{eqnarray}
  A\otimes_{S(V)}K(S(V))=A\otimes_{S(V)} S(V)\otimes_{\C}\bigwedge W_1=A\otimes_{\C}\bigwedge W_1
\end{eqnarray}
%with the differential
%$$
%\sum_ia_i\frac{\partial}{\partial \theta_i},
%$$
%where we denote the basis of $T_1$ by $\theta_1,\dots,\theta_n$. 
%Thus ${\rm Tor}^{S(V)}_p(A,\C)$ is precisely the Koszul homology of $A$ with respect to the sequence of generators $\{a_1,\dots,a_n\}$.

%This gives a second description of the algebra of syzygies: 
Alternatively, suppose we have a minimal free resolution of $A$,
\[\xymatrix{\cdots \rto& F_2 \rto& F_1 \rto& F_0 \rto& A \rto& 0,} \]
%This is an exact sequence 
by graded free $S(V)$-modules
\[F_p=\bigoplus\limits_{q%\geq m_p
}R_{p,q}\otimes S(V).\]
Here minimality means that the differential vanishes on tensoring this complex with
the trivial $S(V)$-module $\C$% (c.f. \cite[Section 2.1]{GKhR})
, and
hence
\[\mathrm{Tor}^{S(V)}_p(A,\C)\:\:=\:\:R_p:=\:\bigoplus\limits_{q
%\geq m_p
}R_{p,q}.\]
Thus
$R_{p,q}$ is the finite dimensional vector space of $p$-th
  order syzygies of degree $q$ for $A$%, and $m_p$ is the minimum degree among the $p$-th order syzygies
.  
%We use this description in the following example.
%\begin{rem}
%When we have this description of the syzygies of $A$, we can informally think of them as iterated relations of $A$. That is, the lowest degree syzygies $R_1$ are the relations, then second degree syzygies $R_2$ are the relations amongst these relations and so on.
%\end{rem}

\begin{eg} \label{eg:G25}
Consider the Pl\"ucker embedding of the Grassmannian $G(2,5)$,
\begin{eqnarray*}
G(2,5) &\longrightarrow& \mathbb{P}\left({\bigwedge}^{\!2}\; \C^5\right)\\
\langle v_1,v_2 \rangle_{\C} &\longmapsto& [v_1\wedge v_2].
\end{eqnarray*}
%We can write this as an intersection of quadrics using the following figure:
%\begin{center}
%\begin{tikzpicture}
%\path (0,0) coordinate (P1);
%\path (P1) ++(0:4) coordinate (P2);
%\path (P2) ++(1*72:4) coordinate (P3);
%\path (P3) ++(2*72:4) coordinate (P4);
%\path (P4) ++(3*72:4) coordinate (P5);

%\draw [name path=1--2](P1) node[below left]{$e_1$} -- (P2);
%\draw [name path=2--3](P2) node[below right]{$e_2$} -- (P3);
%\draw [name path=3--4](P3) node[right]{$e_3$} -- (P4);
%\draw [name path=4--5](P4) node[above]{$e_4$} -- (P5);
%\draw [name path=1--5](P5) node[left]{$e_5$} -- (P1);
%\draw [name path=1--3](P1) -- (P3);
%\draw [name path=1--4](P1) -- (P4);
%\draw [name path=1--5](P1) -- (P5);
%\draw [name path=2--4](P2) -- (P4);
%\draw [name path=2--5](P2) -- (P5);
%\draw [name path=3--5](P3) -- (P5);

%\path [name intersections={of=2--4 and 3--5,by=(G1)}];
%\draw [name intersections={of=(P14) and (P35),by=(G2)}];
%\draw [name intersections={of=(P14) and (P25),by=(G3)}];
%\draw [name intersections={of=(P13) and (P25),by=(G4)}];
%\draw [name intersections={of=(P13) and (P24),by=(G5)}];

%\node [fill=red,inner sep=1pt,label=-90:$\tilde\Gamma_1$] at (G1) {};
%\end{tikzpicture}
%\end{center}
Let $\{e_1,\dots,e_5\}$ be a basis of $\C^5$, then $G(2,5)$ can be written as the intersection of           five 
quadrics $\Gamma_i=0$  where
\begin{eqnarray*}
\Gamma_1&= & -e_{24}e_{35}+e_{23}e_{45}+e_{25}e_{34} \\
\Gamma_2&= &- e_{14}e_{35}+e_{13}e_{45}+e_{15}e_{34} \\
\Gamma_3&= & -e_{14}e_{25}+e_{12}e_{45}+e_{15}e_{24} \\
\Gamma_4&= & -e_{13}e_{25}+e_{12}e_{35}+e_{15}e_{23} \\
\Gamma_5&= & -e_{13}e_{24}+e_{12}e_{34}+e_{14}e_{23}.
\end{eqnarray*}
Here $e_{12}$ is the coordinate function dual to $e_1\wedge e_2$, etc.
The projective coordinate algebra of $G(2,5)$  
is 
$
A=S(V)/I
$
where $V$ is the vector space with basis  $\{e_{ij}\}$ and $I=(\Gamma_1,\dots,\Gamma_5)$.

From \cite[Example 2.3.3]{GKhR} we know the dimensions of the syzygies and we can explicitly construct a minimal free resolution by $S(V)$ modules,
$$
\xymatrix@R=0pt{0\rto&F_3 \rto& F_2 \rto& F_1\rto& F_0\rto& A \rto&0.\\
&\langle c^*\rangle& 
\langle\tilde\Lambda_1,\dots,\tilde\Lambda_5\rangle& 
\langle\tilde\Gamma_1,\dots,\tilde\Gamma_5\rangle& 
\langle c\rangle &S(V)/I }
$$
%From \cite[\S 7]{GrossWallach} the Hilbert series for $G(2,5)$ is
%$$
%H(t)=\frac{1-5t^2+5t^3-t^5}{(1-t)^{10}}
%$$
%and this tells us the dimensions of the syzygies.
%This fits into our framework as a commutative quadratic algebra $S(V)/I$ with $V^*={\rm span}(e_{ij}\;|\; 1\leq i < j \leq 5)$ and $I=(\Gamma_1,\dots,\Gamma_5)$. 

In particular, $F_0=R_{0,0}\otimes S(V)\cong S(V)$
and the zeroth order syzygy $R%_{G(2,5)}
_{0,0}$ is given by a copy of $\C$, generated by $c$.
For exactness at $F_0$ we take $F_1=R_{1,2}\otimes S(V)$ where $R_{1,2}$ is spanned by $\tilde\Gamma_1,\dots,\tilde\Gamma_5$ and 
$\partial \tilde\Gamma_i=\Gamma_i\cdot c$.
The kernel of this boundary map is spanned by
\begin{eqnarray*}\begin{split}
\Lambda_1\;=\;\;&& -e_{12}\tilde\Gamma_2 & +e_{13}\tilde\Gamma_3 & -e_{14}\tilde\Gamma_4 & +e_{15}\tilde\Gamma_5 \\
\Lambda_2\;=\;\;&e_{12}\tilde\Gamma_1 & & -e_{23}\tilde\Gamma_3 & +e_{24}\tilde\Gamma_4 & -e_{25}\tilde\Gamma_5  \\
\Lambda_3\;=\;\;&-e_{13}\tilde\Gamma_1 & +e_{23}\tilde\Gamma_2 &  & -e_{34}\tilde\Gamma_4 & +e_{35}\tilde\Gamma_5\\
\Lambda_4\;=\;\;&e_{14}\tilde\Gamma_1 & -e_{24}\tilde\Gamma_2 &  +e_{34}\tilde\Gamma_3 &  & -e_{45}\tilde\Gamma_5\\
\Lambda_5\;=\;\;&-e_{15}\tilde\Gamma_1 & +e_{25}\tilde\Gamma_2 & -e_{35}\tilde\Gamma_3 & +e_{45}\tilde\Gamma_4 &&\\
\end{split}\end{eqnarray*}
so for exactness at $F_1$ we take  $F_2=R_{2,3}\otimes S(V)$ where $R_{2,3}$ is spanned by $\tilde\Lambda_1,\dots,\tilde\Lambda_5$ and
\begin{eqnarray*}
\partial\tilde\Lambda_1= (0,-e_{12},e_{13},-e_{14},e_{15})\\
\partial\tilde\Lambda_2= (e_{12},0,-e_{23},e_{24},-e_{25})\\
\partial\tilde\Lambda_3= (-e_{13},e_{23},0,-e_{34},e_{35})\\
\partial\tilde\Lambda_4= (e_{14},-e_{24},e_{34},0,-e_{45})\\
\partial\tilde\Lambda_5= (-e_{15},e_{25},-e_{35},e_{45},0)
\end{eqnarray*}
The third and highest order syzygy $R%_{G(2,5)}
_{3,5}$ has a single generator $c^*$ with \[\partial c^*=\sum_{i=1}^5\Gamma_i\,\tilde\Lambda_i.\]

We return to this example in Example \ref{eg:G25alg}, where we write down the algebra structure on these syzygies.
%The differentials in this complex are given by degenerating the basis elements into their syzygy representatives, for example $$\tilde\Gamma_1^*\mapsto -e_{12}\tilde\Gamma_2 +e_{13}\tilde\Gamma_3  -e_{14}\tilde\Gamma_4 +e_{15}\tilde\Gamma_5,
%$$
%from which it is easy to see that it is a complex.

%Further, we can read off the algebra structure of the syzygies ???
%, which is given by the relation:
%$$
%c^*=\sum_{i=1}^5\tilde\Gamma_i\tilde\Gamma^*_i
%$$
%and all other relations being zero.
%Koszul algebras are also be characterised as follows: $S(V)/I$ is Koszul if and only if the matrix of the differential at each stage is linear in $V$. This is certainly the case here and we conclude that $S(V)/I$ is Koszul.
\end{eg}

\subsection{The Berkovits complex}\label{sec:ber}

%We now describe Berkovits' set up for a 
Let $A$ be a commutative quadratic algebra given by $S(V)/I$, where $V$ is a vector space with basis $a_1,\dots,a_n$, and let $W_1$ be the vector space concentrated in degree 1 with basis 
%We consider $A$ as a DG algebra with a trivial differential.
%The first step is to freely adjoin a set of generators
$\theta_1,\dots,\theta_n$.

%, one for each generator $a_i$. Next, we write the minimal set of generators
%$\{\Gamma_1,\dots,\Gamma_m\}$ of the defining ideal $I$
%as homology classes in the semi-free extension of $A$ by $T_1$.

\begin{lem}\label{lem:syzlowest}
  Suppose the quadratic ideal $I$ is spanned by% for a commutative Koszul algebra $A$ are
%  defined by the formulas
$$\Gamma_k=\sum_{i,j=1}^n \Gamma_{ij}^k a_ia_j,\qquad\qquad(1\leq k\leq m).$$
Then the first order syzygies are given by the following homology classes in $A[W_1]$,
%the representative for the homology class in
%the semi-free extension $A[T_1]$ is given by 
$$\tilde \Gamma_k=\sum_{i,j=1}^n \Gamma_{ij}^k a_i\theta_j,\qquad\qquad(1\leq k\leq m).$$
\end{lem}
% \begin{proof}
%   We use the description of the algebra of syzygies of $A$ given by
%   the complex (\ref{eqn:syz}). Firstly, \[d_K(\tilde
%   \Gamma_k)=\sum_{i,j=1}^n \Gamma_{ij}^k a_ia_j\] is zero as an
%   element of $A$ and $\tilde\Gamma_k$ is a cycle for
%   $k=1,\dots,m$. Assume that $\tilde\Gamma_k$ is a boundary, then
%   there exists $\beta_k\in A\otimes\bigwedge^2 V$ such that
%   $d_K(\beta_k)=\tilde\Gamma_k$ for each $k=1,\dots,m$. We can
%   write \[\beta_k=\sum_{i,j=1}^n\lambda^k_{ij}\theta_i\theta_j,\] for
%   $\lambda^k_{ij}\in A$. Using the fact that the $\theta_i$
%   anticommute and calculating $d_K\beta_k$ we arrive at the following
%   relations,
% \begin{eqnarray*}
% \lambda^k_{ij}+\lambda^k_{ji} &=& 0\\
% \lambda^k_{ij}-\lambda^k_{ji} &=& \Gamma^k_{ij}.
% \end{eqnarray*}
% This implies that $\Gamma^k_{ij}=2\lambda^k_{ij}$. Finally, since the
% generators of the ideal are symmetric,
% i.e. $\Gamma^k_{ij}=\Gamma^k_{ji}$, this means $\beta_k=-\beta_k=0$
% for each $k$ giving us our contradiction. Therefore, the sequence
% $\{\tilde\Gamma_1,\dots,\tilde\Gamma_m\}$ are homology classes.
% \end{proof}

This leads to the following definition.

\begin{defn}\label{defn2:ber}
  The \emph{Berkovits complex} of a commutative quadratic
  algebra $A$ is
$$A[W_1\oplus W_2]:=\left(A\otimes\bigwedge W_1 \otimes
  S(W_2),\;
\sum_{i=1}^n
  a_i\frac{\partial}{\partial\theta_i}+\sum_{k=1}^m
  \tilde\Gamma_k\frac{\partial}{\partial y_k}\right),$$
where $W_2$ is the vector space spanned by elements $y_1,\dots,y_m$ in homological degree two.

The \emph{Berkovits homology} of $A$ is the homology of this complex.
\end{defn}

Definitions \ref{defn:koshom} and \ref{defn2:ber}, which calculate Koszul and Berkovits homology respectively, are examples of \emph{semi-free extensions} of commutative DG algebras, or \emph{relative Sullivan algebras} (see for example 
\cite{Avr,hess-hrht,Qui,Sul}):

\begin{defn}
A \emph{semi-free extension} of $A$ is an inclusion of commutative DG algebras $$A\to A\otimes S_\bullet(W)$$  for some positively graded vector space $W$. We will write $A[W]$ for $A\otimes S_\bullet(W)$.
%a graded vector space which is assumed to have a well-ordered basis $\{w_\lambda:\lambda\in\Lambda\}$ and the differential respects the ordering, $\partial w_\lambda\in A\otimes S(W_{<\lambda})$  
\end{defn}
By $S_\bullet$ we mean the symmetric algebra in the graded sense,
$$
S_\bullet(W)=S(W_{\text{even}})\otimes \bigwedge  W_{\text{odd}}.
$$
For instance Definition \ref{defn2:ber} would read $A[W_1\oplus W_2]= A \otimes S_\bullet(W_1\oplus W_2)$.

\subsection{Minimal models for commutative Koszul algebras} 
Assume that $A$ is a commutative Koszul algebra with
$A^!=U(L)$ as above. We construct a resolution of $A$ in the category of DG
algebras from the Chevalley-Eilenberg complex of $L$.

\begin{defn}\label{defn:chev}
  The Chevalley-Eilenberg complex of $L$ is the cochain complex with \[\mathrm{CE}^i(L)=\left({\bigwedge}^i L\right)^*,\] and the
  differential $d_C:\mathrm{CE}^k(L)\to\mathrm{CE}^{k+1}(L)$ is defined on
  $\varphi\in \left({\bigwedge}^k L\right)^*$ by
  \[(d_C\varphi)(x_0,\dots,x_{k}) = \sum_{i<j} (-1)^{j+\varepsilon(i,j)}
  \varphi(x_0,\dots,x_{i-1},[x_i,x_j],x_{i+1},\dots,\widehat{x_j},\dots,x_k),\]
  where each $x_r$ is homogeneous, $\widehat{x_j}$ means omit $x_j$ and
  $\varepsilon(i,j)=|x_j|(|x_{i+1}|+\dots+|x_{j-1}|)$.
\end{defn}
% In particular $d_C:\mathrm{CE}^1(L)\to\mathrm{CE}^{2}(L)$ is the map that
%  is dual to the bracket,\[(d_Cf)(x_0,x_1)=-f([x_0,x_1]).\]
The Chevalley-Eilenberg complex is a cochain complex that calculates the cohomology of $L$ with trivial coefficients, $H^*(L,\C)$. Chevalley and Eilenberg in \cite{ChEil} introduced a {\em shuffle product} $\odot$ with respect to which the differential $d_C$ is a derivation and gives this complex 
an algebra structure that descends to the algebra structure on
the cohomology of $L$. %More explicitly, given 
%$$
%\varphi\in
%\left({\bigwedge}^{\!i} L\right)^* {\rm and} \;\psi\in %\left({\bigwedge}^{\!j}
%  L\right)^* {\rm , then}\; \varphi\odot\psi\in %\left({\bigwedge}^{\!i+j} L\right)^*
%$$
%and $d_C(\varphi\odot\psi)=d_C\varphi\odot\psi\pm\varphi\odot
%d_C\psi$.

For graded Lie superalgebras with finite dimensional graded components, one can prove that the shuffle product $\odot$ is
skew-supersymmetric and that there is an algebra
isomorphism \[\left(\bigwedge L^*,\wedge\right) \;\cong\;
%\longrightarrow
\left(\left(\bigwedge L\right)^*,\odot\right).\]%(see Appendix \ref{appen}) 
Hence, the Chevalley-Eilenberg complex can be alternatively defined as
the exterior algebra on $L^*$, with differential $d_C$ the extension as a
derivation of the map dual to  $[\;,\,]:L\wedge L\to L$.
%\[(d_Cf)(x_0,x_1)=-f([x_0,x_1]).\]
%From now on we will use this definition.

As well as being a cochain complex whose cohomology is that of $L$, the Chevalley-Eilenberg complex may also be considered a chain complex which defines a resolution of $A$. The chain complex is given by defining $L_p^*$ to have homological grading $p-1$, so that the homological and cohomological gradings together give the total degree in $\bigwedge L^*$.
We illustrate the Chevalley-Eilenberg complex as a chain complex with homological grading as follows:
\[\xymatrix@R=0.3pc@!C=6.8em{\cdots\;
\mathrm{CE}_4(L)\rto^-{d_C}& \mathrm{CE}_3(L)\rto^-{d_C}& \mathrm{CE}_2(L)\rto^-{d_C}&\mathrm{CE}_1 (L)\rto^-{d_C}&\mathrm{CE}_0(L)
%\rto&0
\\
&&&0 \ar[r] & L_1^* 
%\ar[r] & 0 
\\
&&0 \ar[r] & L_2^* \ar[r] & \bigwedge^2 L_1^* 
%\ar[r] & 0 
\\
&0 \ar[r] & L_3^* \ar[r] & L_2^* \wedge L_1^* \ar[r] & \bigwedge^3 L_1^* 
%\ar[r] & 0 
\\
0 \ar[r] & L_4^* \ar[r] & L_3^* \wedge L_1^* \oplus L_2^* \wedge L_2^* \ar[r] & L_2^* \wedge L_1^* \wedge L_1^* \ar[r] \ar[r] & \bigwedge^4 L_1^* 
%\ar[r] & 0
}\]
The original cohomological grading is seen on the diagonals:
$\mathrm{CE}^1(L)=L^*$ is the first non-zero diagonal and
$\mathrm{CE}^2(L)$ is the diagonal containing exterior products of two factors.

\begin{defn}
A \emph{minimal model} of a commutative algebra $A=S(V)/I$ is a semi-free extension $S(V)\to S(V)[W]$ such that

\begin{enumerate}
\item  the canonical quotient $S(V)[W]\to S(V)\to A$ is an isomorphism in homology, 
\item the differential is decomposable, $\partial %W_n 
W
\subseteq 
%\left(
S(V)\otimes S_\bullet
^{\geq2} W%_{\leq (n-1)}
%\right)_{n-1}
.$
\end{enumerate}
\end{defn}
This concept arose in rational homotopy theory, see for example \cite[Definition 1.10]{hess-hrht} or \cite[Section 7.2]{Avr}.

\begin{prop}\label{H0L}
  For a commutative Koszul algebra $A$ with $A^!=U(L)$, the Chevalley-Eilenberg
  complex of $L$ with homological grading is a minimal model of $A$.
\end{prop}
\begin{proof}

We observe that $\mathrm{CE}_0(L)=\bigwedge L_1^*=S(V)$. If we take $W_k=L_{k+1}^*$ for $k\geq1$ then $\mathrm{CE}(L)$ is the semi-free extension $S(V)[W]$.

%When $L$ is a graded Lie superalgebra with $A^!=U(L)$
  %The Chevalley-Eilenberg complex of $L$ with cohomological degree given by the number of exterior powers calculates $H^*(L,\C) \cong  \mathrm{Ext}_{U(L)}^*(\C,\C)$. 
  As $L$ is graded, so is
  $\mathrm{Ext}_{U(L)}^*(\C,\C)$, and since $A$ is Koszul \[H^i(L,\C)_j \cong
  \mathrm{Ext}_{U(L)}^{ij}(\C,\C)
  =\left\{
\begin{array}{l l}
  A & \quad \mbox{if $i=j$}\\
  0 & \quad \mbox{if $i\neq j$}\\
\end{array} \right.,\] where $j$ is the total degree.
  %algebra, we
  %have \[A=\bigoplus_{i\geq0}\mathrm{Ext}_{A^!}^{ii}(\C,\C)\] and
  %$\mathrm{Ext}_{A^!}^{ij}(\C,\C)=0$ for $i\neq j$ (c.f. Definition %\ref{defn:koszul}). But $A^!=U(L)$
%  and \[H_i(\mathrm{CE}(L))_j=H^{j-i}(L,\C)_j,\] where on the left
%  hand side we are 
Using our homological grading of
  $\mathrm{CE}(L)$ we have therefore,
\[H_i(\mathrm{CE}(L))=\left\{
\begin{array}{l l}
  A & \quad \mbox{if $i=0$}\\
  0 & \quad \mbox{if $i\neq0$}\\
\end{array} \right.\] and the Chevalley-Eilenberg complex of $L$ is a resolution of $A$.

%Consider the Chevalley-Eilenberg complex with the homological grading, which we have seen is a resolution of $A=S(V)/I$. Let $S:=S(V)=\mathrm{CE}_0(L)$, then $\mathrm{CE}(L)$, with the DG algebra structure constructed above, is a semi-free extension of $S(V)$ by the positive graded set $B$ with $B_i=L_{i+1}^*$.

%We observe that $\textrm{Ch}_0(L)=\bigwedge L_1^*=S(V)$ and we have a canonical quotient map $\mathrm{CE}_0(L)\to A$. 
Finally, the differential is decomposable, since
%$d_C(L_2^*) \subset L^*_1 \wedge L^*_1=S^2(V)$ and 
$%$
d_C(L_n^*)\subseteq \sum_{i=1}^{n-2} L_i^* \wedge L_{n-i}^*.
$ 
%The Chevalley-Eilenberg complex is then a minimal model $S(V)[W]$ for $A$ with $W_k=L_{k+1}^*$ for $k\geq1$.
\end{proof}
\begin{rem}\label{eulerpoincare}
It follows that the Hilbert series of $A$ may also be calculated as the Euler characteristic of the Chevalley-Eilenberg complex with homological grading, and we recover \eqref{gauss}.
\end{rem}

%As established in the work of \cite{Ber, GKhR, MSh}, the cohomology of $L_{\geq 2}$ is a classical object related to $A$ called the algebra of syzygies.

%We have seen in Proposition~\ref{H0L} that $H^*(L,\C)\cong A$. The following result may be found in \cite{MS,her}.

We will now give a proof of the following theorem, based on a more general result of Avramov.
\begin{thm}\label{thm:syz}
The algebra of syzygies of $A$ is isomorphic to $H^*(L_{\geq2},\C)$ and the Berkovits homology is isomorphic to $H^*(L_{\geq3},\C)$.
\end{thm}

%The idea is as follows: given a cycle $z$ that generates some homology class in a DG algebra $(X, \partial)$, we embed $X$ into a DG algebra $X^\prime$ by freely adjoining a variable $y$ of degree $|z|+1$ such that $\partial y=z$. In this way, $z$ is now a boundary and no longer a homology class.

%The construction is dependent on the degree of the cycle:

%\noindent{\bf $|z|$ is even:} \cite[Construction 2.1.7]{Avr} In this case, we take
%$$
%X^\prime = X \otimes \bigwedge (y)
%$$
%with differential
%$$
%\partial(x_0 + x_1y)=\partial(x_0)+\partial(x_1)y+(-1)^{|x_1|}x_1z,
%$$
%for $x_0,x_1\in X$.

%\noindent{\bf $|z|$ is odd:} \cite[Construction 2.1.8]{Avr} In this case, we take
%$$
%X^\prime = X \otimes \C[y]
%$$
%with differential
%$$
%\partial\left(\sum_ix_iy^i\right)=\sum_i\partial(x_i)y^i+(-1)^{|x_i|}ix_izy^{i-1},
%$$
%for $x_i\in X$.

%\begin{defn}
%A semi-free extension of $X$ is a DG algebra $X^\prime$ obtained by repeated adjunction of free variables. When $T$ is the graded set of generators of the free variables, the semi-free extension is denoted $X[T]$.
%\end{defn}

%*****

%In other words, ${\rm Ber}(A)$ is the semi-free extension $A[W_1\oplus W_2]$. Further, it can be seen as two steps in a transfinite process. From our algebra $A$, we can construct a resolution of $\C$ from the augmentation $A\to \C$ by successively adjoining free variables of higher degree in a possibly transfinite manner. By continually picking a set of minimal generators, we obtain the acyclic closure.

Recall first the following definition from  \cite[Section 6.3]{Avr}.

\begin{defn}
The \emph{acyclic closure} of %$\C$ over 
a commutative DG algebra $A$ is a semi-free extension $A \to A[W]$ such that 
$A[W]$ is a resolution of $\C$ and $\{\partial w\;|\; w \in W_{n+1}\}$ minimally generates the reduced homology $\widetilde H_n(A[W_{\leq n}])$ for each $n\geq0$.
\end{defn} 
Thus the complexes calculating the Koszul and Berkovits homology are simply the first two steps of the acyclic closure of $A$. The following result of Avramov may be used to relate a general stage of the acyclic closure to the minimal model of $A$.

Recall that two DG algebras $X$ and $X'$ are quasi-isomorphic if there is a sequence of  quasi-isomorphisms of DG algebras $X \sim X^1 \sim \dots \sim X^n \sim X'$ pointing in either direction.

%We now define the notion of a minimal model and state a theorem that relates the minimal model of a graded algebra $A$ and its acyclic closure.

%\begin{defn}\cite[Section 7.2]{Avr}
%Let $S$ be a graded algebra. A minimal DG algebra over $S$ is a semi-free extension $S[B]$ such that the adjoined variables $B$ are positively graded and the differential $\partial$ is decomposable, in the sense that $\partial(B_1) \subseteq SW^2$, where $W$ is a minimal set of generators for $S$, and $\partial(S[B]) \subseteq (B)^2S[B]$.
%
%If $A=S/Q$ then a minimal model of $A$ over $S$ is a quasi-isomorphism $S[B]\to A$, where $S[B]$ is a minimal DG algebra.
%\end{defn}

%We have the following useful theorem.

\begin{thm}[{\protect\cite[Theorem 7.2.6]{Avr}}]
Suppose $A$ is a commutative DG algebra with acyclic closure
$A[W]$ and minimal model $S(V)[W']$. Then, for each $n\geq 1$, the DG algebras $A[W_{\leq n}]$ and $S(V)[W']/(W'_{< n})$ are quasi-isomorphic.
\end{thm}

If $A$ is Koszul with $A^!=U(L)$ then the Chevalley-Eilenberg complex of $L$ with homological grading is a minimal model $S(V)[W']$ of $A$, where $W_i'=L^*_{i+1}$. Since $S(V)[W']/(W'_{<n})=\mathrm{CE}(L_{\geq n+1})$  the theorem gives a quasi-isomorphism
$$
A[W_{\leq n}]\;\sim\;\mathrm{CE}(L_{\geq n+1}).
$$
Theorem \ref{thm:syz} is just the cases $n=1$ and $2$.

%Using a spectral sequence argument (see \cite[Theorem 3.6.3]{GKhR}) one can in fact show that the algebra structures on $H^*(L_{\geq2},\C)$ and on the algebra of syzygies agree, and further that
%$$
%R_{pq}=H^{q-p}(L_{\geq2},\C)_q,
%$$
%where $q$ on the right-hand side is the total degree in the Chevalley-Eilenberg complex.
%***********************************************************

%*****************************************************

%\section{Free subalgebras}
%OLD INTRODUCTION?? As explained in the introduction from the point of view of the %gauge theory it is important to have free subalgebras in the Lie algebra %$L_{\bullet}$. In \cite{MSh} it was proved that $L_{\geq 3}$ is a free Lie algebra %using quite an involved argument. VERY? OLD: We have found out that such question of %existing of a free subalgebra in a super Lie algebra with a few relations has been %addressed in the mathematical literature. We are grateful to L.A.Bokut for providing %the references. We outline how one proves that $L_{\geq k}$ is free for large enough %$k$.
\section{A smaller complex}\label{sec:bv}

% Must: define "Koszul" and "Koszul complex" of quadratic algebras

Let $A$ be a Koszul commutative algebra, with Koszul dual $A^!$ the universal enveloping algebra of a graded Lie superalgebra $L$. 
We have seen above that the Chevalley-Eilenberg complex of $L$ (with homological grading) is a resolution of $A$, and the Chevalley-Eilenberg complexes of $L_{\geq 2}$ and $L_{\geq 3}$  calculate the algebra of syzygies $R$ and the Berkovits homology respectively.%We achieved this by recasting the problem in terms of the acyclic closure and minimal models in order to apply a theorem of Avramov.

The Berkovits complex is not a module over $\C[y_1,\dots,y_m]$, which makes its homology difficult to calculate in practice. There is a construction for the case of the orthogonal Grassmannian $OG(5,10)$ alluded to in the introduction which appears in Movshev--Schwarz \cite{MSh}, where they found an alternative complex quasi-isomorphic to the Berkovits complex, which is a module over $\C[y_1,\dots,y_m]$. This gives them a more manageable complex and they are able to calculate its homology directly. 

Inspired by this construction, we sketch an alternative complex for calculating the Berkovits homology for a specific example where the algebra of syzygies can be explicitly described. %We also borrow the name `$bv_\mu$' from Movshev--Schwarz.

%We pointed out in Remark \ref{rem:syzygy} (2) that the syzygies of a commutative Koszul algebra carries an $A_\infty$-algebra structure. In order to construct our alternative complex we will use the higher multiplications. We illustrate appearance of an $A_\infty$-algebra with an example.
\subsection{The Grassmannians $G(2,N)$}

To that effect consider the Grassmannian as a homogeneous space ${\rm SL}_N(\C)/P$, where $P$ is some minimal complex parabolic subgroup that is not a Borel. In the usual way, the Pl\"ucker embedding of $G(2,N)$ is given by the orbit of the highest weight vector in a highest weight representation of ${\rm SL}_N(\C)$. The corresponding projective coordinate algebra $A_{G(2,N)}$ is quadratic and Koszul. According to Gorodentsev et al.\ in \cite{GKhR} its algebra of syzygies of $A_{G(2,N)}$ is also quadratic and Koszul, but it was pointed out to us by the referee that this result is incorrect. The algebra of syzygies $R_{G(2,N)}$ in fact admits an non-trivial $A_\infty$-algebra structure for $N>5$, see Example 4.5 below.

Gorodentsev et al.\ give a full set of generators and relations for $A_{G(2,N)}$ in terms of semistandard Young tableaux. We recall some definitions regarding such tableaux and state this result here.

\begin{defn}
A partition $\lambda=(\lambda_1,\dots,\lambda_N)$, where $\lambda_1\geq \lambda_2 \geq \dots \geq \lambda_N\geq0$, is represented by a Young diagram with rows of lengths $\lambda_1,\dots,\lambda_N$. The transposed diagram is denoted $\lambda'=(\lambda'_1,\dots,\lambda'_{N'})$.
A partition $\lambda$  is written in Frobenius notation as
$$
(\alpha_1,\dots,\alpha_p|\beta_1,\dots,\beta_p)\qquad\qquad 
\lambda_i=\alpha_i+i,\;\lambda_i^\prime=\beta_i+i.$$
For example, the partitions $(6,4,3,3,1,0)$ and $(6,4,3,3,1,0)^\prime=(5,4,4,2,1,1)$ are depicted
$$
{\tiny\yng(6,4,3,3,1,0)}\qquad\qquad\qquad\qquad
{\tiny\yng(5,4,4,2,1,1)}
$$
and are written $(5,2,0|4,2,1)$ and $(4,2,1|5,2,0)$ in Frobenius notation.

A semistandard Young tableau is a Young diagram in which the boxes are labelled with numbers $1,\dots,N$ weakly increasing across rows and strictly increasing down columns. For instance,
\begin{align}\label{eqn:yng}
  \begin{aligned}
\tiny\young(112455,3346,455,566,6)
  \end{aligned}
\end{align}
%is a semistandard Young tableau for $N=6$. 
The set of all semistandard Young tableaux of shape $(\alpha_1,\dots,\alpha_p|\beta_1,\dots,\beta_p)$ is denoted
$$
\pi_{(\alpha_1,\dots,\alpha_p|\beta_1,\dots,\beta_p)}.
$$
\end{defn}
For each partition $\lambda$ there is
% set $\pi_\lambda$ corresponds to 
an irreducible ${\rm GL}_N(\C)$-module (and in fact ${\rm SL}_N(\C)$-module), also denoted by $\pi_{(\alpha_1,\dots,\alpha_p|\beta_1,\dots,\beta_p)}$,
obtained as the quotient space
$$\left.\left({\bigwedge}^{\!\lambda_1^\prime}V\otimes\dots\otimes {\bigwedge}^{\!\lambda_{k^\prime}^\prime}V\right)\right/_{^{\text{\small column exchange relations,}}}
$$
where $V\cong\C^N$, the standard module for ${\rm GL}_N(\C)$ or ${\rm SL}_N(\C)$. We do not make these relations explicit, except to note that the semistandard Young tableaux 
in $\pi_{(\alpha_1,\dots,\alpha_p|\beta_1,\dots,\beta_p)}$
define basis vectors in this representation. For example the basis vector for (\ref{eqn:yng}) is
$$
(e_1\wedge e_3 \wedge e_4 \wedge e_5 \wedge e_6) \otimes (e_1\wedge e_3 \wedge e_5 \wedge e_6) \otimes (e_2\wedge e_4 \wedge e_5 \wedge e_6) \otimes (e_4\wedge e_6) \otimes e_5 \otimes e_5,
$$
where $e_1,\dots,e_6$ are a basis of the standard module for ${\rm GL}_6(\C)$ or ${\rm SL}_6(\C)$. Further,
%$\lambda-\mu>0$ is constant for two partitions $\lambda, \mu$ then
$\pi_{(\lambda_1,\dots,\lambda_N)} \cong \pi_{(\lambda_1-\lambda_N,\dots,\lambda_{N-1}-\lambda_N,\,0) }\otimes \mathrm{det}^{\lambda_N}$.

If $\mf h$ is the standard Cartan subalgebra of diagonal matrices in $\mf{gl}_N={\rm Lie}({\rm GL}_N(\C))$, we let $\varepsilon_1,\dots,\varepsilon_N$ be the standard basis of $\mf h^*$. The highest weight of the irreducible representation corresponding to a partition $\lambda$ is $(\lambda_1-\lambda_N)\varepsilon_1+\dots+(\lambda_{N-1}-\lambda_N)\varepsilon_{N-1}$. We also note that the fundamental weights of ${\rm GL}_N(\C)$ are $\omega_i=\varepsilon_1+\dots+\varepsilon_i$ for $i=1,\dots,N-1$. The highest weight of the representation defined by Young diagrams of shape depicted in (\ref{eqn:yng}) is $6\varepsilon_1+ 4\varepsilon_2+ 3\varepsilon_3+ 3\varepsilon_4+ \varepsilon_5$ or $[2,1,0,2,1]$ written in a basis of fundamental weights.

One introduces the algebra
$$
\mathsf{A}(N)= \bigoplus_{p,q} \mathsf{A}_{p,q}(N)
$$
where
$$
\mathsf{A}_{p,q}(N):=\bigoplus_{\substack{N-2\geq i_1>\dots>i_p\geq 2 \\ i_1+\dots+i_p=q}}\pi_{((i_1-2),\dots,(i_p-2)|(i_1+1),\dots,(i_p+1))}.
$$
This algebra is generated by
$$
\mathsf{A}_{1,r}=\pi_{(r-2|r+1)},
$$
for $2 \leq r \leq N-2$. The multiplication is given by the projection onto the relevant irreducible component in the tensor product representation.
%$$
%\pi_{(r_1-2|r_1+1)}\otimes \pi_{(r_2-2|r_2+1)} \supset \pi_{(r_1-2,r_2-2|r_1+2,r_2+1)}.
%$$
We have the following:
\begin{thm}\label{thm:ghkr}\cite[Theorem 4.4.1]{GKhR} The algebra of syzygies of $G(2,5)$ form a bigraded supercommutative Frobenius Koszul algebra isomorphic to $\mathsf{A}(5)$ with
$$
(R_{G(2,5)})_{-p+q,q}=\mathsf{A}_{p,q}(5).
$$
\end{thm}
\begin{rem}
 The theorem as stated in \cite{GKhR} claims that the syzygies of $G(2,N)$ form a bigraded supercommutative Frobenius Koszul algebra. We suggest why this is not true for $N=6$. In particular the existence of a non-trivial ternary product implies that it is not Koszul. 
\end{rem}

\begin{eg}\label{eg:G25alg}
We return to Example \ref{eg:G25} and consider the syzygies of $G(2,5)$, in which case
$$
\begin{array}{ccccccccc}
\mathsf{A}(5) & = & \mathsf{A}_{0,0}(5) & \oplus & \mathsf{A}_{1,2}(5) & \oplus & \mathsf{A}_{1,3}(5) & \oplus & \mathsf{A}_{2,5}(5)\\
& = & \pi_{\emptyset} & \oplus & \pi_{(0|3)} & \oplus & \pi_{(1|4)} & \oplus & \pi_{(10|43)}\\
& = & \C & \oplus & [0,0,0,1] & \oplus & [1,0,0,0] &\oplus & \C \\
& = & (R_{G(2,5)})_{0,0} & \oplus & (R_{G(2,5)})_{1,2} & \oplus & (R_{G(2,5)})_{2,3} & \oplus & (R_{G(2,5)})_{3,5},

\end{array}
$$
where $\pi_\emptyset$ denotes the empty semistandard Young tableau and the corresponding generator in $R_{G(2,5)}$ is given by $c$. The full list of other semistandard Young tableaux along with their generators in $R_{G(2,5)}$ are given by
$$\Yvcentermath1
 \pi_{(0|3)}:\;\;\tilde\Gamma_1=\;\tiny{\young(2,3,4,5)} \;\;\tilde\Gamma_2=\;\tiny{\young(1,3,4,5)} \;\;\tilde\Gamma_3=\;\tiny{\young(1,2,4,5)} \;\;\tilde\Gamma_4=\;\tiny{\young(1,2,3,5)} \;\;\tilde\Gamma_5=\;\tiny{\young(1,2,3,4)}
$$ 
$$\Yvcentermath1
\pi_{(1|4)}:\;\;\tilde\Lambda_i=\;\tiny{\young(1i,2,3,4,5)} \qquad\qquad \pi_{(10|43)}:\;\;c^*=\;\tiny{\young(11,22,33,44,55)}
$$
so that $\mathrm{dim}\pi_{(0|3)}=\mathrm{dim}\pi_{(1|4)}=5$ and $\mathrm{dim}\pi_{(10|43)}=1$. Further, the algebra structure is given by $\tilde\Gamma_i \tilde\Lambda_j =\tilde\Lambda_j \tilde\Gamma_i=(-1)^{i+1}\delta_{ij}c^*$, $c$ is the unit of the algebra and all other multiplications are zero.

%A necessary condition for $\pi_\lambda \otimes \pi_\mu \supset \pi_\nu$ is $|\nu| = |\lambda|+|\mu|$. Therefore, all maps $\mathsf{A}(5)^{\otimes p} \to \mathsf{A}(5)$ for $p>2$ are zero. This shows that the syzygies of $G(2,5)$ form a bigraded supercommutative Frobenius Koszul algebra.
\end{eg}

\begin{eg}
 We consider the syzygies of $G(2,6)$ given by $\mathsf{A}(6)=\mathsf{A}_0(6)\oplus \mathsf{A}_1(6)\oplus \mathsf{A}_2(6) \oplus \mathsf{A}_3(6)$ where
 
$$
\begin{array}{ccccccc}
\mathsf{A}_1(6) & = & \mathsf{A}_{1,2}(6) & \oplus & \mathsf{A}_{1,3}(6) & \oplus & \mathsf{A}_{1,4}(6)\\
& = & \pi_{(0|3)} & \oplus & \pi_{(1|4)} & \oplus & \pi_{(2|5)}\\
& = & [0,0,0,1,0] & \oplus & [1,0,0,0,1] & \oplus & [2,0,0,0,0] \\
& = & (R_{G(2,6)})_{1,2} & \oplus & (R_{G(2,6)})_{2,3} & \oplus & (R_{G(2,6)})_{3,4},
\end{array}
$$

$$
\begin{array}{ccccccc}
\mathsf{A}_2(6) & = & \mathsf{A}_{2,5}(6) & \oplus & \mathsf{A}_{2,6}(6) & \oplus & \mathsf{A}_{2,7}(6)\\
& = & \pi_{(10|43)} & \oplus & \pi_{(20|53)} & \oplus & \pi_{(21|54)}\\
& = & [0,0,0,0,2] & \oplus & [1,0,0,0,1] & \oplus & [0,1,0,0,0] \\
& = & (R_{G(2,5)})_{3,5} & \oplus & (R_{G(2,6)})_{4,6} & \oplus & (R_{G(2,6)})_{5,7},
\end{array}
$$
and $\mathsf{A}_0(6)=\mathsf{A}_3(6)=\C$. %The shapes of the Young diagrams are

%$$\Yvcentermath1
%\pi_{(0|3)}:\tiny{\yng(1,1,1,1)} \qquad \pi_{(1|4)}:\tiny{\yng(2,1,1,1,1)} \qquad \pi_{(2|5)}:\tiny{\yng(3,1,1,1,1,1)}
%$$

%$$\Yvcentermath1
% \pi_{(10|43)}:\tiny{\yng(2,2,2,2,2)} \qquad \pi_{(20|53)}:\tiny{\yng(3,2,2,2,2,1)} \qquad \pi_{(21|54)}:\tiny{\yng(3,3,2,2,2,2)} \qquad \pi_{(210|543)}:\tiny{\yng(3,3,3,3,3,3)}.
%$$
We suggest two reasons why this is not a Koszul algebra. The first is that it is not a hook algebra in the sense of \cite[Section 5.1]{GKhR}: it fails to satisfy $x\cdot y = 0$ for all $x,y\in \pi_{T}$ where $T$ is a partition. Indeed by the Littlewood-Richardson rule $\pi_{(1|4)}\otimes \pi_{(1|4)} \supset \pi_{(20|53)}$ is not zero as claimed.

Further the syzygies appear to have a nontrivial $A_\infty$-algebra structure. Certainly one can see there is a ternary product $m_3:\mathsf{A}_{1,2}(6)\otimes \mathsf{A}_{1,2}(6) \otimes \mathsf{A}_{1,2}(6) \rightarrow \mathsf{A}_{2,6}(6)$ directly by multiplying the Schur functions\footnote{The calculation was performed using SAGE.}
$$
{s_{(0|3)}}^3=s_{(10|54)} + 2s_{(20|53)} + 3s_{(21|52)} + s_{(21|43)} + s_{(210|510)} + 2s_{(210|420)} + s_{(210|321)}.
$$
The second factor is precisely $\mathsf{A}_{2,6}(6)$. If the algebra were Koszul then there could be no such nontrivial map of homological degree $1$ and internal degree $0$. 
\end{eg}

%Corollary \ref{cor:bv} and Theorem \ref{thm:ghkr} mean that we can use the complex $bv_\mu$ to calculate some cohomologies of $L_{\geq 3}$.

\subsection{The Complex}

We now construct the complex that calculates the Berkovits homology for $A_{G(2,5)}$ with $A_{G(2,5)}^!=U(L)$. Initially, the idea was to construct this complex for a general commutative Koszul algebra whose algebra of syzygies is also Koszul. However, as pointed out to us by the referee, this is a rare situation in the real world. We borrow the name `$bv_\mu$' from Movshev--Schwarz. Choose a basis $\{y_1,\dots,y_5\}$ of $L_2$ inside the symmetric algebra $S(L_2)$ and
consider the complex
$$
bv_\mu\;\;=\;\;
S(L_2)\otimes_{\mathsf{A}(5)^!} K(\mathsf{A}(5)^!) \;\;\cong \;\;S(L_2)\otimes_{\mathsf{A}(5)^!} \mathsf{A}(5)^!\otimes_\C \mathsf{A}(5)^*.
$$
%That is,
%\begin{eqnarray}\label{bv}
%&&\!\!\!\!\!\!
%bv_\mu\;\;\cong \;\;S(L_2)\otimes_{R^!} R^!\otimes_\C R^*\;\;\cong\\
%&&\xymatrix{\cdots\rto&S(L_2)\otimes R^*_3\rto
%&S(L_2)\otimes R^*_2\rto&
%S(L_2)\otimes R^*_1\rto&S(L_2)\otimes R^*_0}\nonumber.
%\end{eqnarray}
Here $L_2$ 
 is concentrated in degree 0, and as usual for the Koszul complex we have $\mathsf{A}(5)^!$ concentrated in degree 0 and
$\mathsf{A}(5)^*$ has grading induced by having its generators concentrated in degree 1. In the more general situation, one would have to replace the Koszul resolution by a suitable resolution of the $A_\infty$-algebra of syzygies $R$, such as the bar resolution.

Consider the degree zero map of quadratic algebras
% $R^!$ acts on $S(L_2)$ via the map $R^!\to S(L_2)$ that is Koszul dual to
$$\bigwedge L_2^*\to \mathsf{A}(5),\qquad y_k^*\mapsto \tilde\Gamma_k.$$
The dual map induces a map $\mathsf{A}(5)^!\to S(L_2)$ which gives an action of $\mathsf{A}(5)^!$ on $S(L_2)$ by multiplication. Again, in the more general situation we would have to consider $S(L_2)$ as an $R$-module and this action may be more complicated.

The differential is $1\otimes d$ where $d$ is the Koszul differential of $K(\mathsf{A}(5)^!)$. Thus we may write
$$
bv_\mu=\left( S(L_2)\otimes \mathsf{A}(5)^*
, \;\; \sum_{k=1}^5\; y_k\;\frac\partial{\!\!\partial \tilde\Gamma_k^*}\right).
$$     
In more general cases the differential would have a more complicated form due to the presence of higher multiplications and may contain terms of higher order in $\partial/\partial\tilde\Gamma_k^*$.

%We show that this complex calculates the cohomology of $L_{\geq3}$ by adapting the argument from~\cite[Section 3]{HS} which studies the case where $R^!\cong U(L_{\geq2})$ is the Yang-Mills algebra, and which is also inspired by~\cite{MSh}. 

%The idea is to construct a complex $B^+E_\mu$ whose homology is the same as $bv_\mu$ and $\mathrm{CE}(L_{\geq3})$:

%$$
%bv_\mu\stackrel\sim\longrightarrow  B^+E_\mu 
%\stackrel\sim\longleftarrow \mathrm{CE}(L_{\geq3}).
%$$

Taking into account the natural grading on the symmetric algebra we see that $bv_\mu$ is filtered by the complexes $S^{\geq p}(L_2)\otimes \mathsf{A}(5)^*$
%$$
%F_{-p}=\left(
%\xymatrix@C=1.7em{\cdots\rto&S^{\geq p}(L_2)\otimes R^*_3\rto
%&S^{\geq p}(L_2)\otimes R^*_2\rto&
%S^{\geq p}(L_2)\otimes R^*_1\rto&S^{\geq p}(L_2)\otimes R^*_0}\right)
%$$
and the associated graded is isomorphic to the bigraded vector space
\begin{equation}\label{grbv}
\mathrm{Gr}(bv_\mu)=
S(L_2)\otimes \mathsf{A}(5)^*.
\end{equation}
\noindent In the more general case there is no guarantee that this associated graded object has no differential.

Choose a basis $\{q_1,\dots,q_m\}$ of $L_2$ inside $U(L_{\geq2})$. Now let 
$$E_\mu=\left(\bigwedge L_2\otimes U(L_{\geq2}), \;\;\sum_{k=1}^{m}\;y_k\;\frac{\partial}{\!\!\partial q_k\!\!}\;\;\right)$$
with the standard augmented algebra structure.  This has a filtration $\bigwedge^{\geq p} L_2\otimes U(L_{\geq2})$ whose associated graded $\mathrm{Gr}(E_\mu)$ is the algebra itself but with zero differential.
The reduced bar complex $B^+E_\mu$ has an induced filtration, and we observe that the  
corresponding associated graded complex is isomorphic to the bar complex of the associated graded, 
\begin{equation}\label{gremu}
\mathrm{Gr}(B^+E_\mu)\;\;=\;\;
B^+ (\mathrm{Gr}(E_\mu))\;\;=\;\;
B^+ \left(\bigwedge L_2\;\otimes \;U(L_{\geq2})\right)\;.\end{equation}

We show that $bv_\mu$ calculates the cohomology of $L_{\geq3}$ by adapting the argument from~\cite[Section 3]{HS} which studies the case where $R^!\cong U(L_{\geq2})$ is the Yang-Mills algebra, and which is also inspired by~\cite{MSh}. 

The idea is to use the complex $B^+E_\mu$ as an intermediary whose homology is the same as $bv_\mu$ and $\mathrm{CE}(L_{\geq3})$:
$$
bv_\mu\stackrel\sim\longrightarrow  B^+E_\mu 
\stackrel\sim\longleftarrow \mathrm{CE}(L_{\geq3}).
$$
The map $p:bv_\mu \to B^+E_\mu$ is constructed as the composition of a map $bv_\mu \to E_\mu$ with the quasi-isomorphism $E_\mu \to B^+E_\mu$. If we map $y_k\mapsto y_k$, $\tilde\Gamma_k\mapsto q_k$ and $\tilde\Lambda_k$ to the corresponding generator in $L_3$ inside $U(L_{\geq2})$, then this map clearly respects the filtrations.
\begin{thm}
Homologies of the complexes $bv_\mu$ and $B^+E_\mu$ coincide.
\end{thm}
\begin{proof}
Consider the spectral sequences associated to the filtrations on $bv_\mu$ and $B^+E_\mu$ defined above, 
which converge to their respective homologies and which have zero pages the associated graded objects \eqref{grbv}, \eqref{gremu}.

By Theorem \ref{thm:syz} we can say that there exists a chain map
$$
\mathsf{A}(5)^*\cong \mathsf{A}(5)\stackrel\sim\longrightarrow \mathrm{CE}(L_{\geq 2}) 
\stackrel\sim\longrightarrow B^+(U(L_{\geq 2}))
$$
that induces isomorphism of homology groups, given by some choice of homology classes in the Chevalley--Eilenberg or bar complexes. Using this,
together with classical Koszul duality for symmetric and exterior algebras, and the K\"unneth theorem,
 we can construct a chain map 
%$$\mathrm{Gr}(bv_\mu)
%\stackrel\sim\longrightarrow \mathrm{Gr}(B^+E_\mu).$$ 
%That is, we have a chain map 
$$
S(L_2)\otimes \mathsf{A}(5)^*
\stackrel\sim\longrightarrow 
B^+(\bigwedge L_{2})\otimes B^+(U(L_{\geq 2}))\stackrel\sim\longrightarrow B^+\left(\bigwedge L_{2}\otimes U(L_{\geq 2})\right)
$$
inducing isomorphism in homology. In other words, we have a chain map
inducing isomorphism between the homologies of the associated graded objects above,
$$
\mathrm{Gr}(bv_\mu)\stackrel\sim\longrightarrow \mathrm{Gr}(B^+E_\mu).
$$
This is a weak equivalence between the zero pages of the spectral sequences associated to the filtrations on $bv_\mu$ and $B^+E_\mu$. Hence we have an isomorphism between the first pages. The result follows from the comparison theorem of spectral sequences.
\end{proof}

For the second quasi-isomorphism we adapt the proof in  \cite{HS}.

\begin{prop}
The algebras $E_\mu$ and $U(L_{\geq 3})$ are quasi-isomorphic, or the homologies of $B^+(E_\mu)$ and $\mathrm{CE}(L_{\geq3})$ coincide. The quasi-isomorphism is given by the inclusion $z\in U(L_{\geq3}) \mapsto 1 \otimes z\in E_\mu$.
\end{prop}

\begin{proof}
We must prove that the cohomology 
of $E_\mu$ is $U(L_{\ge3})$. 
Consider the following filtration of $U(L_{\ge2})$:
$$
F^j:=\left\{\begin{array}{cl}
0 & j=0\\
\left\{z\in U(L_{\ge2})\; | \; \frac{\partial}{\partial q_i}(z)\in F^{j-1}\; \forall i\right\} & j>0
\end{array}\right.
$$
Poincar\'e-Birkhoff-Witt implies that $F^1=U(L_{\ge3})$ (see \cite[Proposition 3.1]{HS} for more details) and by \cite[Lemma 28]{MSh} the filtration is multiplicative, exhaustive, Hausdorff ($\bigcap F^j=0$) and $q_i\in F^2$. %Note that \cite[Lemma 29]{MSh} is also true in this context.
We define a filtration of $E_\mu$ as in \cite[Proposition 3.7]{HS}:
$$F_pE_{\mu,q}:= F^{p+5-q}\otimes{\bigwedge}^{5-q}L_2.$$
It is clear that $d(F_pE_{\mu,q}) \subseteq F_{p-2}E_{\mu,q-1}$, where $d$ is the differential on the complex $E_\mu$, and so the differentials on the $E^0$ and $E^1$-pages of the spectral sequence associated to this filtration are zero.

Since \cite[Lemma 29]{MSh} is true in this context, we have that $\textrm{Gr}_F(U(L_{\geq2}))\cong U(L_{\geq3})\otimes \C[\hat{q}_1,\dots,\hat{q}_5]$, where $\hat{q}_1,\dots,\hat{q}_5$ are the images of $q_1,\dots,q_5$ in $\textrm{Gr}_F(U(L_{\geq2}))$. The $E^2$-page is given by
$$
E^2_{pq}=F_pE_{\mu,p+q}/F_{p-1}E_{\mu,p+q} \cong 
U(L_{\geq3})\otimes \C[\hat{q}_1,\dots,\hat{q}_5]\otimes {\bigwedge}^{5-p-q} L_2.
$$
The differential on the $E^2$-page is nothing but the Koszul differential for $\bigwedge L_2$ (see \cite[Proposition 3.7]{HS} for more details), which is a resolution of $\C$. Hence, the spectral sequence collapses on the $E^3$-page to $E^3_{0,5}=U(L_{\geq3})$.
\end{proof}
We have the following immediate corollary.

\begin{cor}\label{cor:bv}
The complex $bv_\mu$ calculates the Berkovits homology. Further, we have $H^{2}(L_{\geq3},\C)= 0$ and so $L_{\geq3}$ is free.
\end{cor}
\begin{proof}
 The first statement is clear.
 For the second statement, we use a result from \cite[Ex II.2.9]{BourLie} that if $L_{\geq3}$ is positively graded then $L_{\geq3}$ is a free Lie (super)algebra if and only if $H^2(L_{\geq3},\C)=0$. The differential in the complex $bv_\mu$ is:
$$
d%_{bv_\mu}
(c)=0,\qquad 
d%_{bv_\mu}
(\tilde\Gamma_i)=y_ic,\qquad 
d%_{bv_\mu}
(\tilde\Lambda_i)=0, \qquad 
d%_{bv_\mu}
(c^*)= \sum_{i=1}^5 (-1)^{i+1}y_i\tilde\Lambda_i.
$$
%The differential of $c^*$ can be deduced from the algebra structure on $R_{G(2,5)}$. 
Since the only contribution to $H^{2}(L_{\geq3},\C)$ could be given by $c^*$, and this is not a cycle, we obtain $H^{2}(L_{\geq3},\C)= 0$.
\end{proof}

\subsection{Further generalization}\label{gen-kos}
%\shortcomment{not done}
%\begin{rem}
%\begin{enumerate}
%\item If the algebra of syzygies $R_A$ of $A$ is quadratic and Koszul, then $R^%!\cong U(L_{\ge 2})$ and so $YM_A$ is quadratic and Koszul also. 
%\item The theorem implies that $L_{\ge 3}$ is free for $A_{G(2,5)}$ but not for% $A_{G(2,N)}$ for $N\ge6$. For this we use the following Hilbert series formula% for, where  $n=N-2$
%$$ \sum_{j=1}^n\frac{\frac{1}{n}\binom{n}{j} \binom{n}{j-1}q^{i-1}} {(1-q)^{2n+%1}}$$
%but make denominator $(1-q)^{\binom{N}{2}}$.

We presented a simple construction of a complex $bv_\mu$ quasi-isomorphic to the Berkovits homology and to the Lie algebra $L_{\geq 3}$, which made certain calculations very straightforward. Although it is only constructed for $G(2,5)$, we would like to investigate whether the construction may be further extended to cases where the algebra of syzygies has higher structure.

We suspect that one may enrich the homology isomorphism between the complexes $bv_\mu$ and $B^+E_\mu$. Namely, one would like quasi-isomorphisms of differential graded {\em algebras}: between the cobar construction $\Omega bv_\mu$, the algebra $E_\mu$, and the algebra $U(L_{\geq3})$.  For this, however, it would be necessary to define an appropriate coalgebra structure on $bv_\mu$.  In the case that the algebra of syzygies $R$ is quadratic, $bv_\mu$ will will be a `strict' differential graded coalgebra, and other cases it will be an $A_\infty$-coalgebra. An explicit definition of such an $A_\infty$ structure was given in the paper \cite{MSh} of Movshev--Schwarz for the example they considered. 

\end{document}